\documentclass[a4paper,10pt]{amsart}

\usepackage{cite}
\usepackage{amssymb}
\usepackage{amsthm}

\usepackage{amsfonts}
\usepackage[centertags]{amsmath}
\numberwithin{equation}{section}

\newtheorem{theorem}{Theorem}[section]
\newtheorem{lemma}{Lemma}[section]
\newtheorem{remark}{Remark}[section]
\newtheorem{definition}{Defintion}[section]

\newcommand{\RR}{{\mathbb R}}

\newcommand{\QQ}{\mathcal{Q}}
\newcommand{\DD}{\mathcal{D}}
\newcommand{\p}{\prime}

\newcommand{\pt}{\partial_t}
\newcommand{\pr}{\partial_r}
\newcommand{\ppt}{\partial^2_t}
\newcommand{\ppr}{\partial^2_r}
\newcommand{\pR}{\partial_R}
\newcommand{\ppR}{\partial^2_R}
\newcommand{\pa}{\partial_a}
\newcommand{\ppa}{\partial^2_a}

\newcommand{\bbeta}{{\tilde{\beta}}}
\newcommand{\mtr}{g_{\scriptscriptstyle \Omega \Omega}(r)}
\newcommand{\mt}{g_{\scriptscriptstyle \Omega \Omega}}
\newcommand{\dmt}{\dot{g}_{\scriptscriptstyle \Omega \Omega}}

\def\skwdrv{(\partial_{\tau}-2\beta(\tau) \xi \partial_{\xi})}

\begin{document}

\title{Blow up for critical wave equations on curved backgrounds}

\author{J. Nahas}
\address{\'Ecole Polytechnique F\'ed\'erale de Lausanne \\
MA B1 487\\
CH-1015 Lausanne}
\email{joules.nahas@epfl.ch}

\author{S. Shahshahani}
\address{\'Ecole Polytechnique F\'ed\'erale de Lausanne \\
MA B1 477\\
CH-1015 Lausanne}
\email{sohrab.shahshahani@epfl.ch}

\maketitle

\begin{abstract}
We extend the slow blow up solutions of Krieger, Schlag, and Tataru to semilinear wave equations on a curved background. In particular, for a class of manifolds $(M,g)$
we show the existence of a family of blow-up solutions with finite energy norm to the equation 
\begin{equation}
\partial_t^2 u - \Delta_g u = |u|^4 u,
\notag
\end{equation}
with a continuous rate of blow up. In contrast to the case where $g$ is the Minkowski metric, the argument used to produce these solutions can only obtain blow up rates that are bounded above.
\end{abstract}


\section{Introduction}
We study the nonlinear focusing wave equation on a curved, three dimensional background manifold $(M, g)$,
\begin{equation}
\partial_t^2 u - \Delta_g u = |u|^4 u,
\label{slv-m}
\end{equation}
where $u: M \times \mathbb R \rightarrow \mathbb R$. Much is known about well posedness and blow up of solutions to this equation when the metric $g$ is flat. Some exciting work--for example \cite{Fon}, \cite{Met-Tay}, \cite{Ank-Pier-Vall}, and references therein, has been done for hyperbolic backgrounds, but to our knowledge little is known about it in the case of a more general metric. We construct a continuum of blow up solutions for this equation such that the energy norm,
\begin{equation}
\|u\|_{\dot{H^1}}^2 \equiv \int_{M} (|\partial_t u|^2 + \langle \nabla u, \nabla u\rangle_{g})\,dvol,
\notag 
\end{equation}
is finite,
like the work of Krieger, Schlag, and Tataru in \cite{Krg-Sch-Ttr}. The scenario of solutions that blow up with bounded energy is referred to as
type {II} blow up. It is widely believed that type {II} blow up solutions are intimately connected to the time independent, finite energy solutions of an equation, the soliton solutions. Since blow up is a local phenomenon, and $M$ is approximately flat near the point $r=0$, we also base our construction on perturbations of the rescaled soliton of the equation from the Minkowski space, rather than from $M$. This soliton can be thought of as a function on the tangent space to $M$ at $r=0$, identified with a neighborhood of the point via the chosen coordinates. We will denote it by
\begin{equation}
W(|x|) = \frac{1}{\sqrt{ 1+\frac{|x|^2}{3} }}.
\notag 
\end{equation}
That the relevant time independent solution has the flat space as background is similar to the bubbling off of harmonic maps from $\mathbb R^2$, in the study of the harmonic map heat flow on compact Riemann surfaces by Struwe in \cite{Str-hrm-ht}. Using finite speed of propagation for the wave equation, we are able to modify this 'local' flat-background solution, to a solution on $M \times \mathbb R$.

Our main result is the following theorem:
\begin{theorem}
\label{main}
Let $(M,g)$ be a smooth three dimensional manifold, and $(r(p), \theta (p),\phi(p))$ be a coordinate chart in some open set $U \subseteq M$ such that the metric $g$ satisfies
\begin{align}
g(r,\theta,\phi) = \,dr^2 + \mtr(\,d\theta^2 + \sin^2 \theta \,d\phi^2),
\notag
\end{align}
where $\mtr$ is analytic in $U$, and obeys the estimate
\begin{align}
\label{mtr-cnd}
\left |\frac{\dot{g}_{\scriptscriptstyle \Omega \Omega}(r)}{\mtr}-\frac{2}{r} \right | \lesssim r.
\end{align}
Then for $\nu \in (1/2, 1]$,
 and arbitrary $\delta > 0$,
there exists a solution $u \in H^{1+\nu}(M)$ to \eqref{slv-m} such that for small enough $t$, $u \in C_0(U)$, and 
\begin{align}
u(r) = t^{-1-\nu} W(t^{-1-\nu} r) + \varepsilon(r,t),
\label{slt-pls-rdt} \\
\textrm{ with }
\int_{r \leq t}(\frac{1}{2}|\nabla \varepsilon|^2 + \frac{1}{2}|\partial_t \varepsilon|^2 + \frac{1}{6}|\varepsilon|^6)\sqrt{\det g}\,dr \,d\theta \,d\phi \rightarrow 0,
\notag \\
\textrm{ and }
\int_{r > t}(\frac{1}{2}|\nabla u|^2 + \frac{1}{2}|\partial_t u|^2 + \frac{1}{6}|u|^6) \sqrt{\det g}\,dr \,d\theta \,d\phi \leq \delta.
\notag
\end{align}
\end{theorem}
These solutions are of considerable interest, because it is conjectured that for the set of all initial data that lead to blow up for the nonlinear wave equation, type {II} solutions are the boundary of this set (see Remark 1.2 of \cite{Dck-Kng-Mrl}). This remains a challenging problem to investigate.
\begin{remark}
Nontrivial examples of a pair $(M,g)$ that satisfies \eqref{mtr-cnd} include the three sphere and three hyperbolic space.
\end{remark}
\begin{remark} Let $R_{rrrr}(0)$ denote the diagonal radial component of the curvature tensor as $r \rightarrow 0$. It is a straightforward exercise to show that for small enough $r$, the difference between $g$ and the flat metric, $ r^2 - \mtr $, satisfies
\begin{equation}
r^2 - \mtr   = \frac{1}{3} R_{rrrr}(0) r^4 + o(r^4). 
\notag 
\end{equation}
\label{scl-brk}
\end{remark}
\begin{remark}  In contrast to \cite{Krg-Sch-Ttr}, where there was no upper bound on the rate of the blow up, our methods only produce blow up for $\nu  \leq 1$. Whether these rates can be extended above this threshold is an interesting open question.
\end{remark}
We expect the corresponding local well posedness, global well posedness, and blow up results developed for the flat case to hold for this equation, which we briefly recall. Global well posedness for initial data in the energy space up to the energy of the first soliton, $W(|x|)$, was proven in \cite{Kng-Mrl} using the concentration compactness techniques developed in \cite{Bhr-Grr}. This result is optimal, in the sense that for any energy above this threshold, there exist blow up solutions, see \cite{Krg-Sch-Ttr} (and also \cite{Lvn}).

Duyckaerts and Merle obtained a type {II} blow up classification in \cite{Dck-Mrl} for the semilinear wave equation. They showed that up to the symmetries of the equation, for solutions with free energy equal to that of $W(|x|)$, there is only one solution other than $W(|x|)$ that does not scatter. This solution, $W^{-}(x,t)$, scatters backwards in time, and converges exponentially to $W(|x|)$ forwards in time. This was followed by a result of Duyckaerts, Kenig, and Merle in \cite{Dck-Kng-Mrl} and \cite{Dck-Kng-Mrl-rxv} that for the nonlinear wave equation, type {II} blow up solutions must have a profile that is a rescaled soliton plus radiation, assuming the free energy of the initial data is close enough to that of $W(|x|)$. This fundamental theorem for type {II} blow up solutions allows one to focus on determining the rescaling parameter $\lambda(t)$ and the radiation term.

This construction of slow blow up has been adapted to a number of different situations. It was originally produced for the charge one equivariant wave maps from $\mathbb R^{2+1}$ to a sphere by Krieger, Schlag, and Tataru in \cite{Krg-Sch-Ttr-wv-mps} and wave maps to a surface of rotation by C\^{a}rstea in \cite{crs}, then for the focusing nonlinear wave equation in \cite{Krg-Sch-Ttr}, and for the Yang-Mills equation in \cite{Krg-Sch-Ttr-yng-mll}. The construction has also been modified to produce blow up at infinity for the nonlinear wave equation in three dimensions by Donninger and Krieger in \cite{Dnn-Krg}, and the range the blow up exponent $\nu$ was extended by Krieger and Schlag in \cite{Krg-Sch}. We also note that in an upcoming paper \cite{Shr}, the second author has further extended this machinery to the case of a wave map between two dimensional spheres, with a similar upper bound on $\nu$.

While the solutions produced by this machinery are not smooth, Rodnianski and Sterbenz in \cite{Rdn-Str} prove smooth type {II} blow up for $k$-equivariant wave maps to a sphere, where $k \ge 4$. Moreover, these solutions are stable within the equivariant class of data. This was extended to all $k \ge 1$, as well as equivariant solutions to the Yang-Mills equation by Rodnianski and Rapha\"{e}l in \cite{Rdn-Rph}. The result corresponding to the nonlinear wave equation was proven by Hillairet and Rapha\"{e}l in \cite{Hll-Rph}.

From our ansatz \eqref{slt-pls-rdt}, we treat the laplacian term in \eqref{slv-m} as a flat laplacian plus a perturbation.
Just as in \cite{Krg-Sch-Ttr}, we cannot solve \eqref{slv-m} directly for $\varepsilon$, but must take several renormalization steps before performing a perturbation step. We must modify the procedure from \cite{Krg-Sch-Ttr} to account for an extra term arising from the curved laplacian $\Delta_g$. It is this extra term, see Lemma \ref{crv-lps}, which bounds from above our blow up rates in the perturbation step.

In what follows we first introduce some notation, describe the renormalization steps in Section 2, and prove the main result in Section 3, which mostly consists of the perturbation step.

\subsection{Notation}

To solve \eqref{slv-m}, we will use the assumptions in Theorem \ref{main} to pick a coordinate chart $U$ where the metric has the desired form \eqref{mtr-cnd}. Choose $r_0 > 0$ so that the set
\begin{equation}
U^{\prime} = \{(r, \theta, \phi) \hspace{10pt} \vert \hspace{10pt} r < r_0\}
\end{equation}
is a compactly contained open subset of $U$, and let $\varphi(r) \in C_0^{\infty}(B_{2r_0}(0))$ be such that $\varphi(r) = 1 $ for $r \leq r_0$.
The conserved 'energy' of \eqref{slv-m} will be denoted by
\begin{equation}
E(u) = \frac{1}{4\pi}\int_{M} \Big(\frac{1}{2}|\partial_t u|^2 + \frac{1}{2}|\nabla u|^2 - \frac{1}{6}|u|^6\Big)  \,dvol
\notag
\end{equation}
but for functions compactly supported in $U^{\prime}$, this will become
\begin{equation}
E(u) = \int_{\mathbb R^+} \Big(\frac{1}{2}|\partial_t u|^2 + \frac{1}{2}|\partial_r u|^2 - \frac{1}{6}|u|^6\Big) \mtr \,dr.
\notag
\end{equation}

Instead of solving \eqref{slv-m} directly, we will first solve
\begin{equation}
\partial_t^2 u - \partial_r^2u - \frac{2}{r}\partial_r u-\varphi(r)\left (\frac{\dot{g}_{\scriptscriptstyle \Omega \Omega}(r)}{\mtr}-\frac{2}{r} \right )\partial_r u = |u|^4 u
\label{slv-m-frs}
\end{equation}
for $r \ge 0$, then modify these solutions to satisfy \eqref{slv-m}. For convenience, we define
\begin{equation}
\kappa(r) = \varphi(r)\left (\frac{\dot{g}_{\scriptscriptstyle \Omega \Omega}(r)}{\mtr}-\frac{2}{r} \right )\frac{1}{r}.
\notag 
\end{equation}

\section{Renormalization step}

Our aim in this section is to prove the following theorem.
\begin{theorem}\label{renormalization}
Given $N\geq1$ there is an approximate solution $u_{2k-1}$ to the equation $-\ppt u+\triangle_g u +u^5=0,$ which has the form
\[u_{2k-1}(t,r)=\lambda^{1/2}(t)\left[W(R)+\frac{1}{(t\lambda)^2}O(R)\right],\]
such that the corresponding error $e_{2k-1}:=-\ppt u_{2k-1}+\triangle_g u_{2k-1} +u_{2k-1}^5$ satisfies
\[\int_{r\leq t}|e_{2k-1}(r,t)|^2r^2dr=O(t^N)\quad\quad t\rightarrow 0+.\]
Here the $O(\cdot)$ are uniform in $r\in[0,t]$ and $t\in(0,t_0)$ for a small fixed $t_0,$ and $\lambda=t^{-1-\nu},~R=\lambda r.$
\end{theorem}
\begin{remark}
Recall that $\triangle_{g}$ is given by $\ppr+\frac{\dmt}{\mt}\pr.$ Note that this construction is local, and in particular restricted to the light cone at $t_0$. By choosing $t_0$ sufficiently small, we may therefore assume that we are in the region where $\varphi(r)\equiv1,$ and ignore $\varphi$ in (\ref{slv-m-frs}).
\end{remark}
We start by giving an outline of the strategy. The idea is to construct approximate solutions $u_k$ by iteratively adding correction terms $v_k$ so that $u_k=u_{k-1}+v_k.$ The starting point will be the rescaled soliton $u_0=\lambda^{1/2}W(\lambda r),$ where $\lambda=\lambda(t)=t^{-1-\nu}.$ The error at step $k$ will be
\[e_k=\Box_gu_k+u_k^5.\]
If we let $\varepsilon=u-u_{k-1}$ and linearize the corresponding equation around $\varepsilon=0$ and replace $u_{k-1}$ by $u_0$ (or in other words look at the approximate solution for the linearized operator around $u_0$) we get
\[\big(-\ppt+\ppr+\frac{\dmt}{\mt}\pr+5u_0^4\big)\varepsilon+e_{k-1}\approx 0.\]
We split this up into two cases. If $r\ll t$ we ignore the time derivative (which we expect to be smaller) and replace the linearized equation by
\[\big(\ppr+\frac{\dmt}{\mt}\pr+5u_0^4\big)\varepsilon+e_{k-1}\approx0.\]

If $r\approx t$ we expect $u_0$ to be small and replace it by zero in the linearized equation to get
\[\big(-\ppt+\ppr+\frac{\dmt}{\mt}\pr\big)\varepsilon+e_{k-1}\approx 0.\]
We will use these two simplified equations to construct our approximate solutions. This will be made more precise shortly, but first we need some notation and definitions. The following notation will be used throughout this section.
\begin{align*}
&N_{2k}(v)=\sum_{j=0}^4{5\choose  j}u_{2k-1}^j v^{5-j},\\
&N_{2k+1}(v)=5(u_{2k}-u_0)v+\sum_{j=0}^{3}{5\choose3}u_{2k}^j v^{5-j},\\
&K(v)=\kappa(r)r\pr v,\\
&\lambda=t^{-1-\nu},\\
&R=\lambda r,\\
&a=\frac{r}{t},\\
&b_1=(t\lambda)^{-1}=t^{\nu},\\
&b_2=r^2,\\
&b_3=(t^\nu\lambda)^{-2}=t^2,\\
&B=(b_1,b_2,b_3),\\
&\beta_0=\frac{\nu-1}{2}>-\frac{1}{2},\\
&\mathcal{C}_0=\{(r,t)\big|0\leq r<t,~0<t<t_0\}.
\end{align*}

$b_2=b_3a^2$ and is not technically necessary, but we keep it as an extra variable because it arises naturally. As such $b_2$ and $b_3$ can be thought of as a gain of $t^2$ and $b_1$ as a gain of $t^{2\nu}.$ Note that since our calculations will be done in the light cone $\mathcal{C}_0,$ there are constants $B_1,B_2,$ and $B_3$ such that $b_j\in[0,B_j],$ and $a\in[0,1].$ Let
\begin{align*}
\Omega:=[0,1]\times[0,\infty)\times[0,B_1]\times[0,B_2]\times[0,B_3],
\end{align*}
and denote the projection of $\Omega$ onto the last four factors by $\Omega_a,$ and the projection onto the last three factors by $\Omega_{a,R}.$ Also, keep in mind that
\begin{align*}
5u_0^4(t,R)=\frac{45\lambda^2}{(3+R^2)^2}.
\end{align*}
\begin{definition}\label{Qdef}
$\mathcal{Q}$ is the algebra of continuous functions $q:[0,1]\rightarrow \RR$ with the following properties:
\begin{description}
\item[]  $~~$ (i) q is analytic in $[0,1)$ with an even expansion at $0.$
\item[]  $~~$ (ii) Near $a=1$ we have an absolutely convergent expansion of the form
    \begin{align*}
    q=q_0(a) +\sum_{i=1}^{i=\infty}(1-a)^{\beta(i)+1}\sum_{j=0}^{\infty}q_{i,j}(a)(\log(1-a))^j
    \end{align*}
    with analytic coefficients $q_0,~q_{i,j},$ such that for each $i$ only finitely many of the $q_{i,j}$ are not identically equal to zero. The exponents $\beta(i)$ are of the from
    \[\sum_{k\in K}\left(\left(2k-3/2\right)\nu-1/2\right)+\sum_{k\in K^\p}\left(\left(2k-1/2\right)\nu-1/2\right),\]
    where $K$ and $K^\p$ are finite sets of positive integers.
\end{description}
\end{definition}
\begin{definition}
$\mathcal{Q}^\p$ is the space of continuous functions $q:[0,1]\rightarrow \RR$ with the following properties:
\begin{description}
\item[]  $~~$ (i) q is analytic in $[0,1)$ with an even expansion at $0.$
\item[]  $~~$ (ii) Near $a=1$ we have an absolutely convergent expansion of the form
    \begin{align*}
    q=q_0(a) +\sum_{i=1}^{i=\infty}(1-a)^{\beta(i)}\sum_{j=0}^{\infty}q_{i,j}(a)(\log(1-a))^j
    \end{align*}
    with analytic coefficients $q_0,~q_{i,j},$ such that for each $i$ only finitely many of the $q_{i,j}$ are not identically equal to zero. $\beta(i)$ are as above.
\end{description}
\end{definition}
\begin{definition}
$S^m(R^k(\log R)^l)$ is the class of analytic functions $v:[0,\infty)\rightarrow \RR$ with the following properties:
\begin{description}
\item[]$~~$ (i) $v$ vanishes of order $m$ at $R=0$ and $v(R)=R^m\sum_{j=0}^{j=\infty}c_jR^{2j}$ for small $R.$
\item[]$~~$ (ii) $v$ has a convergent expansion near $R=\infty,$
    \begin{align*}
    v=\sum_{i=0}^\infty \sum_{j=0}^{l+i}c_{ij}R^{k-i}(\log R)^j.
    \end{align*}
\end{description}
\end{definition}
Finally,
\begin{definition}
a) $S^m(R^k(\log R)^l,\mathcal{Q})$ is the class of analytic functions $v:\Omega\rightarrow \RR$ so that
\begin{description}
\item[]$~~$(i) $v$ is analytic as a function of $R,b_i$
                \[v:\Omega_a\rightarrow\mathcal{Q}\]
\item[]$~~$(ii) $v$ vanishes to order $m$ at $R=0$ and is of the form
                \[v\approx R^m\sum_{j=0}^{j=\infty}c_j(a,b_1,b_2,b_3)R^{2j}\]
            around $R=0.$
\item[]$~~$(iii) $v$ has a convergent expansion at $R=\infty,$
            \begin{align*}
            v(\cdot, R, b_1,b_2,b_3)=\sum_{i=0}^\infty \sum_{j=0}^{l+i}c_{ij}(\cdot,b_1,b_2,b_3)R^{k-i}(\log R)^j
            \end{align*}
         where the coefficients $c_{i}:\Omega_{a,R}\rightarrow \mathcal{Q}$ are analytic with respect to $b_1,b_2,b_3.$
\end{description}

b) $IS^m(R^k(\log R)^l, \mathcal{Q})$ is the class of analytic functions $w$ on the cone $\mathcal{C}_0$ which can be represented as
    \[w(t,\alpha)=v(a, R, b_1,b_2,b_3),\quad v\in S^m(R^k(\log R)^l,\mathcal{Q}).\]
\end{definition}
Note that this representation is in general not unique. We are now ready to start the proof of the theorem.
\begin{proof}
We prove that the corrections $v_k$ can be chosen so that they and the corresponding errors $e_k$ satisfy
\begin{align}
&v_{2k-1}\in\sum_{j=0}^{k-1}\frac{\lambda^{1/2}}{(t^\nu\lambda)^{2j}(t\lambda)^{2(k-j)}}IS^2(R(\log R)^{m_k},\QQ)\label{vodd}\\
&t^2e_{2k-1}\in\sum_{j=0}^{k-1}\frac{\lambda^{1/2}}{(t^\nu\lambda)^{2j}(t\lambda)^{2(k-j)}}IS^0(R(\log R)^{p_k},\QQ^\p)\label{eodd}\\
&v_{2k}\in\sum_{j=0}^{k-1}\frac{\lambda^{1/2}}{(t^\nu\lambda)^{2j}(t\lambda)^{2(k+1-j)}}IS^2(R^3(\log R)^{p_k},\QQ)\label{veven}\\
&t^2e_{2k}\in\sum_{j=0}^{k-1}\frac{\lambda^{1/2}}{(t^\nu\lambda)^{2j}(t\lambda)^{2(k-j)}}\big[IS^0(R^{-1}(\log R)^{q_k},\QQ)\nonumber\\&\quad\quad\quad\quad\quad\quad\quad\quad\quad\quad\quad\quad+b_1^2IS^0(R(\log R)^{q_k},\QQ^\p)\nonumber\\&\quad\quad\quad\quad\quad\quad\quad\quad\quad\quad\quad\quad+\sum_{i=2}^3 b_iIS^0(R(\log R)^{q_k},\QQ^\p)\big]\label{eeven}.
\end{align}
Here $m_k,p_k,q_k$ are integers whose exact value is not important for us, and which satisfy $m_1=p_1=0,~q_1=1.$ The two step construction of the correction terms here corresponds to the different approximations of the linearized equation in the regimes $r\ll t$ and $r\approx t$ which we hinted at earlier. It is convenient for future use to introduce the operator $\DD:=\frac{1}{2}+r\partial_r=\frac{1}{2}+R\partial_R.$
\\ \\
$\mathbf{Step~ 0:}$\\
\begin{align*}
e_0&=u_0^5-\Box u_0\\
   &=-\pt\left[\lambda^{\frac{1}{2}}\left(\frac{\lambda^\p}{\lambda}\right)\DD W\right]+\kappa(r)r\pr\left[\lambda^{\frac{1}{2}}W\right]\\
   &=\lambda^{\frac{1}{2}}\left[-\left(\frac{\lambda^\p}{\lambda}\right)^\p\DD W-\left(\frac{\lambda^\p}{\lambda}\right)^2\DD^2W+\kappa(r)R\partial_R W\right].\\
\end{align*}
Therefore, there are constants $c_1,c_2,$ and $c_3,$ depending on $\nu$ and whose exact value is irrelevant for us, such that
\begin{align*}
t^2e_0&=\lambda^{1/2}\left[c_1\left(\frac{1-R^2/3}{(1+R^2/3)^\frac{3}{2}}\right)+c_2\left(\frac{9-30R^2+R^4}{(1+R^2/3)^{\frac{5}{2}}}\right)+\frac{c_3t^2\kappa(r)R^2}{(1+R^2/3)^{\frac{3}{2}}}\right]\\
      &\in \lambda^{\frac{1}{2}}\left(IS^0(R^{-1})+t^2\kappa(r)IS^0(R^{-1})\right)\subseteq \lambda^{\frac{1}{2}}IS^0(R^{-1}).
\end{align*}
$\mathbf{Step ~1:}$\\
Write $e_{2k-2}=\sum_je_{2k-2,j}$ where
\begin{align*}
t^2e_{2k-2,j}\in\frac{\lambda^{1/2}}{(t^\nu\lambda)^{2j}(t\lambda)^{2(k-1-j)}}\big[&IS^0(R^{-1}(\log R)^{q_{k-1}},\QQ)\\
&+\sum_{i=1}^3 b_iIS^0(R(\log R)^{q_{k-1}},\QQ^\p)\big].
\end{align*}
If $e_{2k-2,j}=e_{2k-2,j}(t,a,B,R),$ we let $e_{2k-2,j}^0(t,a,R):=e_{2k-2,j}(t,a,0,R),$ and $e_{2k-2,j}^1=e_{2k-2,j}-e_{2k-2,j}^0$ (unless $k=1$ in which case we let $e^0_{0}=e_0$). We also let $e_{2k-2}^0=\sum_je_{2k-2,j}^0$ and similarly for $e_{2k-1}^1.$ Define $v_{2k-1,j}$ by taking $t$ and $a$ as parameters and requiring that $v_{2k-1,j}$ solve the following ODE in $R,$ subject to vanishing boundary conditions at $R=0:$
\begin{align*}
(t\lambda)^2\Big(-\ppR-\frac{2}{r}\pR-\frac{45}{(3+R^2)^2}\Big)v_{2k-1}=t^2e_{2k-2,j}^0.
\end{align*}
It follows from Lemma 2.1 in \cite{Krg-Sch} that
\[v_{2k-1,j}\in\frac{\lambda^{1/2}}{(t^\nu\lambda)^{2j}(t\lambda)^{2(k-j)}}IS^2(R(\log R)^{m_k},\QQ).\]
Letting $v_{2k-1}=\sum_jv_{2k-1,j}$ we see that (\ref{vodd}) is satisfied.\\\\
$\mathbf{Step ~2:}$\\
The error from the previous step is
\[e_{2k-1}=e^1_{2k-2}+N_{2k-1}(v_{2k-1})+K(v_{2k-1})+E^tv_{2k-1}+E^av_{2k-1},\]
where $N_{2k-1}$ and $K$ are defined above, $E^tv_{2k-1}$ contains the terms in $\ppt v_{2k-1}$ where no derivatives fall on $a$ and $E^av_{2k-1}$ the terms in $(\ppt-\ppr-\frac{2}{r}\pr)v_{2k-1}$ where at least one derivative falls on $a.$ We study $t^2 e_{2k-1}$ term by term, and for this we work on the level of $v_{2j-1,j}$ and sum over $j$ at the end.\\\\
$e^1_{2k-2}$ belongs to the right hand side of (\ref{eodd}) because the $b_j$s contribute the necessary gain of time decay. Indeed, for $b_1$ and $b_3$ this follows from the definition of these variables, and for $b_2$ from the observation that $b_2=t^2a^2.$ \\\\
For the other terms in the error, write $v_{2k-1,j}=t^\bbeta w_{2k-1,j}(a,R)$ where $\bbeta=\bbeta(k,j)$ is defined by $t^\bbeta=\frac{\lambda^{1/2}}{(t^\nu\lambda)^{2j}(t\lambda)^{2(k-j)}},$ and $w_{2k-1,j}\in IS^2(R(\log R)^{m_k},\mathcal{Q}).$ For $t^2E^t$ we just need to note that
\[t^2\pt(t^\bbeta IS^2(R(\log R)^{m_k}))\subseteq t^\bbeta IS^2(R(\log R)^{m_k}).\]
\linebreak
To simplify the notation, for $t^2E^a$ we drop the indices $2k-1$ and $j$ and write $w= w_{2k-1,j}.$ Note that $t^2E^a$ is then a linear combination the following terms
\[t(\pt t^\bbeta)aw_a,~t^\bbeta aw_a,~t^\bbeta aRw_{aR},~t^\bbeta a^{-1}Rw_{aR},~t^\bbeta(1-a^2)w_{aa}.\]
That $t^2E^a$ has the right form follows from the fact that $a\pa,~a^{-1}\pa,~(1-a^2)\ppa$ map $\QQ$ to $\QQ^\p$ .\\\\
With the same notation as before we have
\[t^2 \kappa r\pr v =t^2 \kappa t^\bbeta aw_a+t^2 \kappa t^\bbeta Rw_R,\]
which implies that $K(v_{2k+1})$ is of the right form.\\\\
For $N_{2k-1}(v_{2k-1})$ we work on the level of $v_{2k-1}$ (rather than $v_{2k-1,j}$) and begin by noting that $u_{2k-2}-u_0\in\frac{\lambda^{1/2}}{(t\lambda)^2}IS^2(R(\log R)^n,\QQ),$ for some integer $n$ depending on $k.$\\\\
Linear term (in $v_{2k+1}$): note that
\begin{align*}
u_{2k-2}^4-u_0^4=&(u_{2k-2}-u_0)^4+4(u_{2k-2}-u_0)^3u_0\\
                 &+6(u_{2k-2}-u_0)^2u_0^2+4(u_{2k-2}-u_0)^3u_0.
\end{align*}
We compute (suppressing $\QQ$ for simplicity of notation)
\begin{align*}
t^2(u_{2k-2}-u_0)^4v_{2k-1}&\in \sum_{j=0}^{k-1}\frac{t^{\bbeta(k,j)}(t\lambda)^2}{(t\lambda)^8}IS^2(R(\log R)^{m_k})IS^8(R^4(\log R)^{4n})\\
                         &\subseteq\sum_{j=0}^{k-1}\frac{t^\bbeta}{(t\lambda)^6}IS^{10}(R^5(\log R)^{p_k})\\
                         &\subseteq\sum_{j=0}^{k-1}t^\bbeta a^6IS^4(R^{-1}(\log R)^{p_k})\\
                         &\subseteq \sum_{j=0}^{k-1}t^\bbeta IS^4(R^{-1}(\log R)^{p_k}).
\end{align*}
Similarly,
\begin{align*}
t^2(u_{2k-2}-u_0)^3u_0v_{2k-1}&\in\sum_{j=0}^{k-1}\frac{t^\bbeta(t\lambda)^2}{(t\lambda)^6}S^0(R^{-1})S^2(R(\log R)^{m_k})S^6(R^3(\log R)^{3n})\\
                            &\subseteq \sum_{j=0}^{k-1}\frac{t^\bbeta}{(t\lambda)^4}S^8(R^3(\log R)^{p_k})\subseteq \sum_{j=0}^{k-1}t^\bbeta S^4(R^{-1}(\log R)^{p_k}),
\end{align*}
and
\begin{align*}
t^2(u_{2k-2}-u_0)^2u_0^2v_{2k-1}&\in\sum_{j=0}^{k-1}\frac{(t\lambda)^2t^\bbeta}{(t\lambda)^4}S^0(R^{-2})S^4(R^2(\log R)^{2n})S^2(R(\log R)^{m_k})\\
                              &\subseteq\sum_{j=0}^{k-1}\frac{t^\bbeta}{(t\lambda)^2}S^6(R(\log R)^{p_k})\subseteq \sum_{j=0}^{k-1}t^\bbeta S^4(R^{-1}(\log R)^{p_k}).
\end{align*}
Finally,
\begin{align*}
t^2(u_{2k-2}-u_0)u_0^3v_{2k-1}&\in\sum_{j=0}^{k-1}\frac{t^\bbeta(t\lambda)^2}{(t\lambda)^2}S^0(R^{-3})S^2(R(\log R)^{m_k})S^2(R(\log R)^n)\\
                              &\subseteq\sum_{j=0}^{k-1} t^\bbeta S^4(R^{-1}(\log R)^{p_k}).
\end{align*}
Quintic term:
\begin{align*}
t^2v_{2k-1}^5&\in\sum_{j_1\dots j_5=0}^{k-1}\frac{t^2\lambda^{5/2}IS^{10}(R^5,\QQ)}{(t^\nu\lambda)^{2(j_1+\dots+j_5)}(t\lambda)^{10k-2(j_1+\dots+j_5)}}\\
             &\subseteq\sum_{j,j_1,\dots j_4=0}^{k-1}\frac{a^6t^{\bbeta(k,j)}IS^4(R^{-1},\QQ)}{(t^\nu\lambda)^{2(j_1+\dots+j_4)}(t\lambda)^{8(k-1)-2(j_1+\dots+j_4)}}\\
             &\subseteq\sum_{j,j_1,\dots j_4=0}^{k-1}a^6b_1^{8(k-1)-4(j_1+\dots+j_4)}b_3^{(j_1+\dots+j_4)}t^{\bbeta(k,j)}IS^4(R^{-1},\QQ)\\
             &\subseteq\sum_{j=0}^{k-1}t^{\bbeta(k,j)}IS^4(R^{-1},\QQ).
\end{align*}
Quadratic term: note that $u_{2k-2}\in \lambda^{1/2}IS^0(R^{-1})$ (we are suppressing the algebra $\QQ$ again).
\begin{align*}
t^2u_{2k-2}^3v_{2k-1}^2&\in\sum_{j,l=0}^{k-1}\frac{(t\lambda)^{2}\lambda^{\frac{1}{2}}IS^0(R^{-3})IS^4(R^2(\log R)^{2m_k})}{(t^\nu\lambda)^{2(j+l)}(t\lambda)^{2(2k-l-j)}}\\
                     &\subseteq \sum_{j,l=0}^{k-1}t^{\bbeta(k,j)}b_3^lb_1^{2(k-1-l)}IS^4(R^{-1}(\log R)^{p_k},\QQ)\\
                     &\subseteq\sum_{j=0}^{k-1}t^{\bbeta(k,j)}IS^4(R^{-1}(\log R)^{p_k},\QQ).
\end{align*}
Cubic term:
\begin{align*}
t^2u_{2k-2}^2v_{2k-1}^3&\in\sum_{j_1,j_2,j_3=0}^{k-1}t^2t^{\bbeta(k,j_1)}t^{\bbeta(k,j_2)}t^{\bbeta(k,j_3)}\lambda IS^0(R^{-2})IS^6(R^3(\log R)^{3m_k})\\
                     &\subseteq \sum_{j,j_1,j_2=0}^{k-1}b_3^{(j_1+j_2)}b_1^{2(2k-1-j_1-j_2)}t^{\bbeta(k,j)}IS^6(R(\log R)^{p_k}),
\end{align*}
which has the right form. The quartic term is similar.\\\\
$\mathbf{Step~3:}$\\
Near $R=\infty$ we isolate the principal part $\tilde{e}_{2k-1}$ of $e_{2k-1}$ by ignoring terms that involve a factor of $b_1^2,~b_2$ or $b_3$ or decay at least as fast as $R^{-1}(\log R)^{p_k+2}.$ Let us be more precise and write $t^2\tilde{e}_{2k-1}=t^2\sum_{j=0}^{k-1}\tilde{e}_{2k-1,j},$ where
\begin{align}
t^2\tilde{e}_{2k-1,j}&=t^{\bbeta(k,j)}\Big[\sum_{i=0}^{p_k}q_{i,j}(a)R(\log R)^i+\sum_{i=0}^{p_k+1}{\tilde{q}}_{i,j}^1(a)(\log R)^{i}\nonumber\\
                      &\quad\quad\quad\quad\quad\quad\quad\quad b\sum_{i=0}^{p_k}{\tilde{q}}_{i,j}^2(a)R(\log R)^i+b\sum_{i=0}^{p_k+1}{\tilde{{\tilde{q}}}}_{i,j}(a)(\log R)^i \Big]\nonumber\\
                     &=(t\lambda)t^{\bbeta(k,j)}\Big{[}\sum_{i=0}^{p_k}aq_{i,j}(a)(\log R)^i+b\sum_{i=0}^{p_k+1}({\tilde{q}}_{i,j}^1(a)+a{\tilde{q}}_{i,j}^2)(\log R)^{i}\nonumber\\
                     &\quad\quad\quad\quad\quad\quad\quad\quad+b^2\sum_{i=0}^{p_k+1}{\tilde{{\tilde{q}}}}_{i,j}(a)(\log R)^i\Big{]},\label{aRHS}
\end{align}
for some $q_{i,j},\tilde{q}_{i,j}^{1,2},\tilde{\tilde{q}}_{i,j}\in\QQ^\p,$ with $\tilde{q}_{p_k+1,j}^2=0$ (to be precise $q_{i,j},$ etc, depend on $k$ as well, but we gloss over this to simplify the notation). We want to define $\tilde{v}_{2k,j}$ by
\[t^2(-\ppt+\ppr+\frac{2}{r}\pr)\tilde{v}_{2k,j}=-t^2\tilde{e}_{2k-1,j}.\]
Comparing with (\ref{aRHS}) we seek a solution of the form $\tilde{v}_{2k,j}(t,a)=t^{\bbeta-\nu} w_{2k,j}(a)$ where $w_{2k,j}$ has the form
\begin{align}
w_{2k,j}=&\sum_{i=0}^{p_k}W^i_{2k,j}(a)(\log R)^i+b\sum_{l=1,2}\sum_{i=0}^{p_k+1}{\tilde{W}}_{2k,j}^{i,l}(\log R)^{i}\nonumber\\
                     &+b^2\sum_{i=0}^{p_k+1}{\tilde{{\tilde{W}}}}_{2k,j}^i(a)(\log R)^i.\label{WRHS}
\end{align}
We match the powers of the logarithms in (\ref{aRHS}) and (\ref{WRHS}) to obtain the following equations for $W^i_{2k,j}$
\begin{align*}
&t^2\left(-\ppt+\ppr+\frac{2}{r}\pr\right)\left(t^{\bbeta-\nu} W^i_{2k,j}(a)\right)=t^{\bbeta-\nu} (aq_{i,j}(a)-F_{i,j}(a)),\\
&t^2\left(-\ppt+\ppr+\frac{2}{r}\pr\right)\left(t^\bbeta \tilde{W}^{i,l}_{2k,j}(a)\right)=t^\bbeta (a^{l-1}\tilde{q}_{i,j}(a)-\tilde{F}^l_{i,j}(a)),~l=1,2,\\
&t^2\left(-\ppt+\ppr+\frac{2}{r}\pr\right)\left(t^{\bbeta+\nu} \tilde{\tilde{W}}^i_{2k,j}(a)\right)=t^{\bbeta+\nu} (\tilde{\tilde{q}}_{i,j}(a)-\tilde{\tilde{F}}_{i,j}(a)).
\end{align*}
Here $F_{i,j}$ (and similarly $\tilde{F}^{l}_{i,j,}$ and $\tilde{\tilde{F}}_{i,j}$) can be determined form equations (\ref{aRHS}) and (\ref{WRHS}), and depends only on $a,~W^{i+1}_{2k,j},~\partial_aW^{i+1}_{2k,j}$ and $W^{i+2}_{2k,j},$ where we are using the convention $W^{i}_{2k,j}=0$ for $i>p_k$ (and $\tilde{W}^{i,l}_{2k,j,}=\tilde{\tilde{W}}^i_{2k,j}=0$ for $i>p_k+1$). Note that we are again suppressing the $k$ dependency in the notation. After some rearrangement, this can be written as a system of equations in the variable $a$ as
\begin{align}
&L_\beta W^{i}_{2k,j}=aq_{i,j}(a)-F_{i,j}(a),\quad\quad \beta=\bbeta-\nu,\nonumber\\
&L_\beta \tilde{W}^{i,l}_{2k,j}=a^{l-1}\tilde{q}^l_{i,j}(a)-\tilde{F}^l_{i,j}(a),\quad\quad \beta=\bbeta,\label{lbetaeqn}\\
&L_\beta \tilde{\tilde{W}}^{i}_{2k,j}=\tilde{\tilde{q}}_{i,j}(a)-\tilde{\tilde{F}}_{i,j}(a),\quad\quad \beta=\bbeta+\nu,\nonumber
\end{align}
where $L_\beta$ is defined by
\[L_\beta=(1-a^2)\ppa+2(a^{-1}+\beta a-a)\pa-\beta^2+\beta.\]
It is proved in equation (2,70) in \cite{Krg-Sch} that (\ref{lbetaeqn}) can be solved with zero Cauchy data at $a=0$ yielding solutions which satisfy
\begin{align}
&W^{i}_{2k,j}\in a^3\QQ,\quad\quad i=0,\dots,p_k,\nonumber\\
&\tilde{W}^{i,l}_{2k,j}\in a^{l+1}\QQ,\quad\quad i=0,\dots,p_k+1,~l=1,2,\label{wclaim}\\
&\tilde{\tilde{W}}^{i}_{2k,j}\in a^2\QQ,\quad\quad i=0,\dots,p_k+1.\nonumber
\end{align}
We are now almost ready to define the next correction $v_{2k}.$ However, we need to modify the expression for $\tilde{v}_{2k,j}$ to ensure an even expansion in $R$ and eliminate the singularity of $\log R$ at $R=0.$ With the notation $\langle R\rangle=\sqrt{1+R^2},$ we let
\begin{align}
v_{2k,j}=t^{\beta(k,j)}&\Bigg(R^2\langle R\rangle^{-1}\sum_{i=0}^{p_k}a^{-1}W^{i}_{2k,j}\left(\frac{1}{2}\log(1+R^2)\right)^i\nonumber\\&~+\sum_{i=0}^{p_k+1}\tilde{W}^{i,l}_{2k,j}\left(\frac{1}{2}\log(1+R^2)\right)^i\nonumber\\
&~+bR^2\langle R\rangle^{-1}\sum_{i=0}^{p_k+1}a^{-1}\tilde{W}^{i,2}_{2k,j}\left(\frac{1}{2}\log(1+R^2)\right)^i\nonumber\\&~+b\sum_{i=0}^{p_k+1}\tilde{\tilde{W}}^i_{2k,j}\left(\frac{1}{2}\log(1+R^2)\right)^i\Bigg),\label{vevendef}
\end{align}
and define $v_{2k}=\sum_{j=0}^{k-1}v_{2k,j}.$ That (\ref{veven}) is satisfied is a consequence of (\ref{wclaim}).\\\\
$\mathbf{Step~4:}$\\
The error from the previous step is
\[e^1_{2k-1}+K(v_{2k})+N_{2k}(v_{2k})+(e_{2k-1}^0-(-\ppt+\ppr+\frac{2}{r}\pr)v_{2k}),\]
where $K$ and $N_{2k}$ are defined above, and $e^1_{2k-1}=e_{2k-1}-e^0_{2k-1}$ with $e^0_{2k-1}=\sum_{j=0}^{k-1}e^0_{2k-1,j}$ and
\begin{align*}
t^2e^0_{2k-1,j}=t^{\bbeta(k,j)}&\Bigg[\sum_{i=0}^{p_k}q_{i,j}\frac{R^2}{\langle R\rangle}\left(\frac{1}{2}\log(1+ R^2)\right)^i+\sum_{i=0}^{p_k+1}{\tilde{q}}_{i,j}^1\left(\frac{1}{2}\log(1+ R^2)\right)^i\nonumber\\
                      & b\sum_{i=0}^{p_k+1}{\tilde{q}}_{i,j}^2\frac{R^2}{\langle R\rangle}\left(\frac{1}{2}\log(1+ R^2)\right)^i+b\sum_{i=0}^{p_k+1}{\tilde{{\tilde{q}}}}_{i,j}\left(\frac{1}{2}\log(1+ R^2)\right)^i \Bigg].
\end{align*}
Again we treat the error term by term. $e^1_{2k-1}$ contains one set of terms having a factor of $b_1^2,~b_2$ or $b_3,$ which belong to the right hand side of (\ref{eeven}) by construction, and another set coming from subtracting the leading terms in the expansion in $R$ near $R=\infty.$ To control the latter, notice that after subtracting the first two leading terms, the next term grows more slowly (decays faster) by a factor of $R^2,$ so in view of (\ref{eeven}) it suffices to show that
\begin{align}
IS^0(R^{-1}(\log R)^{q_k},\QQ^\p)\subseteq IS^0(R^{-1}(\log R)^{q_k},\QQ)+b_1^2IS^0(R(\log R)^{q_k},\QQ^\p).\label{ISinclusion}
\end{align}
For this, we decompose an element $w\in IS^0(R^{-1}(\log R)^{q_k},\QQ^\p)$ as $w=(1-a^2)w+a^2w=(1-a^2)w+\frac{R^2}{(t\lambda)^2}w,$ implying
\[w\in IS^0(R^{-1}(\log R)^{q_k},\QQ)+b_1^2IS^2(R(\log R)^{q_k},\QQ^\p)\]
as desired.\\\\
We consider $t^2\left(e_{2k-1,j}^0-(-\ppt+\ppr+\frac{2}{r}\pr)v_{2k,j}\right)$ next. This would be zero if we replaced $\frac{R^2}{\sqrt{1+R^2}}$ by $R$ and $\frac{1}{2}\log(1+R^2)$ by $\log R.$ This means that the so this error involves terms where at least one $t$ or $r$ derivative falls on $R\langle R\rangle,$ or on the difference of the logarithmic terms. The analysis is similar for the two cases, and as the $\log R$ terms are already treated in \cite{Krg-Sch} we focus on the former. In view of (\ref{vevendef}), for the contribution of $\frac{\partial_r}{r}$ we need to look at
\begin{align*}
\frac{t^2t^{\bbeta-\nu}W^{i}_{2k,j}\left(\frac{1}{2}\log(1+R^2)\right)^i\partial_r\left(\frac{R}{\sqrt{1+R^2}}\right)}{r}&=\frac{t^{\bbeta-2\nu}a^{-1}W^{i}_{2k,j}\left(\frac{1}{2}\log(1+R^2)\right)^i}{(1+R^2)^{3/2}}\\
&\in \frac{t^\bbeta R^2 \left(\frac{1}{2}\log(1+R^2)\right)^i}{(1+R^2)^{3/2}}\QQ,
\end{align*}
where we have used (\ref{wclaim}) for the last step. Using (\ref{ISinclusion}) this can be placed in the right hand side of (\ref{eeven}). For the contribution of $\partial_r^2,$ we consider the cases when only one or both $r$ derivatives fall on $R\langle R\rangle.$ The former is taken care of just as the case of $\frac{1}{2}\partial_r$ above. For the latter note that
\[\ppr\big(\frac{R}{\sqrt{1+R^2}}\big)=\frac{-3\lambda^2R}{(1+R^2)^{5/2}}=\frac{-3R^3}{r^2(1+R^2)^{5/2}},\]
which is treated as before. The contribution of $\partial_t^2$ can be dealt with similarly.\\\\

To control $t^2K(v_{2k,j})=t^2\kappa r\pr v_{2k,j}$ we again use the representation (\ref{vevendef}). Noting that $a\partial_a$ sends $\QQ$ to $\QQ^\p$ we can place this contribution in $b_2IS^0(R(\log R)^{q_k},\QQ^\p),$ which is consistent with (\ref{eeven}).\\\\
It remains to consider $N_{2k}(v_{2k}).$ First note that
\[u_{2k-1}-u_0\in\frac{\lambda^{1/2}}{(t\lambda)^2}IS^2(R(\log R)^n,\QQ),\]
for some positive integer $n,$ so
\begin{align*}
u_{2k-1}&\in\lambda^{1/2}W(R)+\frac{\lambda^{1/2}}{(t\lambda)^2}IS^2(R(\log R)^n,\QQ)\\
        &\subseteq \lambda^{1/2}S^0(R^{-1})+\frac{\lambda^{1/2}}{(t\lambda)^2}IS^2(R(\log R)^n,\QQ).
\end{align*}
The contribution of the quintic term can be computed as
\begin{align*}
t^2v_{2k}^5&\in\sum_{j_1,\dots, j_5=0}^{k-1}t^2t^{\bbeta(k+1,j_1)+\dots+\bbeta(k+1,j_5)}IS^{10}(R^{15}(\log R)^{5p_k},\QQ)\\
           &\subseteq \sum_{j,j_2,\dots, j_5=0}^{k-1}\frac{t^{\bbeta(k,j)}(1+R^2)^{16/2}IS^{10}(R^{-1}(\log R)^{q_k},\QQ)}{(t\lambda)^{16}b_1^{-8(k-1)+2(j_1+\dots+j_4)}b_3^{-(j_1+\dots+j_4)}}\\
           &\subseteq \sum_{j=0}^{k-1}t^{\bbeta(k,j)}IS^0(R^{-1}(\log R)^{q_k},\QQ).
\end{align*}
For the linear term we have
\begin{align*}
t^2u_{2k-1}^4v_{2k}&\in (t\lambda)^2\Big(S^0(R^{-1})+a^2IS^0(R^{-1}(\log R)^n,\QQ)\Big)^4\\&\quad\quad\quad\times\sum_{j=0}^{k-1}t^{\bbeta(k+1,j)}IS^2(R^3(\log R)^{p_k},\QQ)\\
                   &\subseteq \sum_{j=0}^{k-1}\frac{\lambda^{1/2}(t\lambda)^2}{(t^\nu\lambda)^{2j}(t\lambda)^{2(k-j)+2}}IS^0(R^{-4}(\log R)^{4n},\QQ)IS^2(R^3(\log R)^{p_k},\QQ)\\
                   &\subseteq \sum_{j=0}^{k-1}\frac{\lambda^{1/2}}{(t^\nu\lambda)^{2j}(t\lambda)^{2(k-j)}}IS^2(R^{-1}(\log R)^{q_k},\QQ).
\end{align*}
The quadratic, cubic, and quartic terms can be treated similarly.
\end{proof}

\section{Perturbation step}

After the renormalization step we must now solve the equation
\begin{align}
\partial_{tt} \varepsilon - \frac{1}{r^2}\partial_r ( r^2 \partial_r \varepsilon) - 5 \lambda^2(t) W^4(\lambda(t) r) \varepsilon =
N_{2k-1}(\varepsilon) + K(\varepsilon) + e_{2k-1}.
\label{prt-qtn}
\end{align}
We change variables several times in order to work with coordinates better adapted to the self similar nature of our solution. Let $\varepsilon(t,x) = v(\tau(t),\lambda(t)x)$, $y = \lambda x$, $\tau$ be a time coordinate that satisfies $\dot{\lambda} = \frac{\,d \lambda}{\,d\tau}$, and
\begin{equation}
\partial_{t} \varepsilon(t,r) = \tau^{\prime}(t)(v_{\tau}+\dot{\lambda}\lambda^{-1}y\partial_y v).
\notag 
\end{equation}
Letting $Rv = \tilde{\varepsilon}$, and $\mathcal{L} = -\partial_{R}^2 - 5W^4(R)$, we obtain the equation
\begin{align}
&(\partial_{\tau} + \dot{\lambda}\lambda^{-1}R\partial_R)^2\tilde{\varepsilon}
-\dot{\lambda}\lambda^{-1}(\partial_{\tau} + \dot{\lambda}\lambda^{-1}R\partial_R)\tilde{\varepsilon}
+\mathcal{L}\tilde{\varepsilon}
 =
\notag \\
& \quad \lambda^{-2}R[
N_{2k-1}(R^{-1}\tilde{\varepsilon})
+K(R^{-1}\tilde{\varepsilon})+e_{2k-1}].
\label{bnd-m}
\end{align}
By transforming this equation so that $\mathcal{L}$ is diagonal, \eqref{bnd-m} will become something close enough to a transport equation to solve. We now recall some of the spectral properties of $\mathcal{L}$. 

On the domain
\begin{equation}
\textrm{Dom}(\mathcal{L})
= \{f \in L^2((0,\infty)) : f, f^{'} \in \textrm{AC}([0,R]) \hspace{10pt} \forall R, f(0)=0, f'' \in L^2((0,\infty))\}
\notag
\end{equation}
the operator $\mathcal{L}$ is self-adjoint. Therefore there exists a unitary operator $\mathcal{F}$ that diagonalizes $\mathcal{L}$. The spectrum of $\mathcal{L}$ is continuous for $\xi   \ge 0$, and has a lone discrete negative eigenvalue, which we will denote by $\xi_{-}$. This discrete negative eigenvalue necessitates some comments about notation. 
The derivative with respect to the spectral variable $\partial_{\xi}$ will be understood to be $0$ for this discrete part.
The continuous part of the spectral measure we denote by $\rho$. From the WKB ansatz, one expects the eigenfunctions $f_{\xi}(R)$ corresponding to $\xi >0$ to have asymptotic form
\begin{equation}
f_{\xi}(R) \simeq e^{-iR\sqrt{\xi}}.
\label{smp-frm}
\end{equation}
Define
\begin{equation}
\mathcal{K} = \mathcal{F} R \partial_R \mathcal{F}^{-1} - \xi \partial_{\xi}.
\label{k-dfn}
\end{equation}
Upon applying $\mathcal{F}$ to $\eqref{bnd-m}$, and using \eqref{k-dfn} several times, we obtain
\begin{align}
& \left (
\begin{array}{cc}
\partial_{\tau}^2 + \xi_{-} & 0 \\
0 & \skwdrv^2+\xi \\
\end{array}
\right )\mathcal{F}\tilde{\varepsilon}
=
\beta(I - 2 \mathcal{K})\skwdrv \mathcal{F}\tilde{\varepsilon}
\notag \\
& \quad - \beta^2 (\mathcal{K}^2 - \mathcal{K} + 2 [\xi \partial_{\xi}, \mathcal{K}])\mathcal{F}\tilde{\varepsilon}
 + \lambda^{-2} \mathcal{F} R (N_{2k-1}(R^{-1} \tilde{\varepsilon})+ \lambda^{-2} \mathcal{F} R (K(R^{-1} \tilde{\varepsilon})+e_{2k-1}).
\label{gg-m}
\end{align}
We will solve \eqref{gg-m} with a contraction mapping argument with the norm
\begin{equation}
\|u\|_{L^{\infty,N}L_{|\xi|^{s/2}\rho \,d\xi}^2}
=
\sup_{\tau > 1} \tau^N \left ( |u(\xi_{-},\tau)| +
\left ( \int_{\mathbb R^+} |u(\xi,\tau)|^2|\xi|^{s/2} \rho \,d\xi \right )^{1/2} \right ).
\label{wgh-spc}
\end{equation}
By using the definition of $\mathcal{L}$, along with $\mathcal{F}$, it is easy to prove the following result from \cite{Krg-Sch-Ttr} 
 about the equivalence of \eqref{wgh-spc} with a weighted Sobolev norm.
\begin{lemma}
\begin{equation}
\|\frac{1}{R}\mathcal{F}^{-1}u\|_{L^{\infty,N}H^s} \lesssim \|u\|_{L^{\infty,N}L_{|\xi|^{s/2}\rho \,d\xi}^2} \lesssim \|\frac{1}{R}\mathcal{F}^{-1}u\|_{L^{\infty,N}H^s}.
\notag
\end{equation}
\label{qvl-nrm}
\end{lemma}
Solving for the discrete component of $\tilde{\varepsilon}$ requires only elementary techniques, so we focus on the continuous components.

The operator $\mathcal{H}$ will denote the solution map for the equation
\begin{equation}
(\skwdrv^2+\xi)u = f, \hspace{10pt} \lim_{\tau \rightarrow \infty}u(\tau)=0,
\notag 
\end{equation}
which has the following smoothing property (see Corollary 6.3 from \cite{Krg-Sch-Ttr}).
\begin{lemma}
\label{smt-ffc}
There exists a $C_0$ not depending on $N$ such that
\begin{equation}
\| \mathcal{H}b\|_{L^{\infty,N-2}L_{\rho}^{2,1/2+\alpha}}
+
\| (\partial_{\tau}-2\beta(\tau) \xi \partial_{\xi})\mathcal{H}b\|_{L^{\infty,N-1}L_{\rho}^{2,\alpha}}
\leq
C_0\frac{1}{N}\|b\|_{L^{\infty,N}L_{\rho}^{2,\alpha}}.
\notag
\end{equation}
\end{lemma}
This smoothing effect is sufficient to counter the loss of regularity, quantified in the following lemmas, from the terms on the right hand side of \eqref{bnd-m}.

We would like to use Proposition 6.7 from \cite{Krg-Sch-Ttr} to handle the term $N_{2k-1}$, however our $N_{2k-1}$ differs from that of the flat background case, due to a slight difference in the renormalization sequence $\{u_{2k-1}\}$. We quote the proposition despite this, since our $\{u_{2k-1}\}$ obey enough of the same properties as in the flat case, that the proof for the proposition is still valid.
\begin{lemma}
\label{nnl-lps}
Assume that $N$ is large enough and $\frac{1}{8} \leq \alpha < \frac{\nu}{4}$. Then the map
\begin{equation}
\mathcal{F}{\tilde{\varepsilon}} \rightarrow \lambda^{-2}\mathcal{F}R(N_{2k-1}(R^{-1}\tilde{\varepsilon}))
\notag
\end{equation}
is locally Lipschitz from $L^{\infty,N-2}L_{\rho}^{2,\alpha+1/2}$ to $L^{\infty,N}L_{\rho}^{2,\alpha}$.
\end{lemma}
From \eqref{smp-frm}, we expect
that the operator $\mathcal{K}$ is bounded. This is proven in \cite{Krg-Sch-Ttr}: 
\begin{lemma}
\label{k-bnds}
a.) The operator $\mathcal{K}$ maps
\begin{equation}
\mathcal{K}: L_{\rho}^{2, \alpha} \rightarrow L_{\rho}^{2, \alpha}.
\end{equation}
b.) In addition, we have the commutator bound
\begin{equation}
[ \mathcal{K}, \xi \partial_{\xi} ]: L_{\rho}^{2, \alpha} \rightarrow L_{\rho}^{2, \alpha}.
\end{equation}
Both statements hold for all $\alpha \in \mathbb R$.
\end{lemma}
The $K$ term is specific to the curved background in our problem, which creates an upper bound on the blow up rates.
\begin{lemma}
Let $\nu \in (0,1]$. Then the map
\begin{equation}
\mathcal{F}\tilde{\varepsilon} \rightarrow
\lambda^{-2}\mathcal{F}RK(R^{-1}\tilde{\varepsilon})
\notag
\end{equation}
is locally Lipschitz from $L^{\infty,N-2}L_{\rho}^{2,\alpha+1/2}$ to $L^{\infty,N}L_{\rho}^{2,\alpha}$.
\label{crv-lps} 
\end{lemma}
\begin{proof}
This term becomes problematic, the intuition being as follows. The arguments to produce these blow up solutions heavily rely on the scale invariance of the energy of the equation.
 The $\kappa$ term breaks the scale since it is a  curvature term (see Remark \ref{scl-brk}).
By Lemma \ref{qvl-nrm}, it suffices to prove the result for the weighted Sobolev spaces, and we further change to cartesian coordinates which facilitates the use of these spaces.
We also recall the estimate from \cite{Tlr-bk} that
\begin{equation}
\|(1-\Delta)^{s/2}(fh)-f(1-\Delta)^{s/2}(h)-h(1-\Delta)^{s/2}(f)\|_{2} \leq
\|f\|_{\infty}\|(1-\Delta)^{s/2} h\|_{2}.
\label{prd-rl}
\end{equation}
Let $R = \sqrt{\sum_{i=1}^{3}X_i^2}$. Using \eqref{prd-rl}, and the support of $\kappa$, 
\begin{align}
\|\lambda^{-3}R& \kappa(\frac{R}{\lambda})\lambda\partial_{R}(R^{-1}\tilde{\varepsilon})\|_{L_{\tau}^{N-2}H^{2\alpha}}\\
& =
\|\lambda^{-2}\kappa(\frac{R}{\lambda}) \sum_{i=1}^{3}\varphi(R/(2\lambda)) X_i\partial_{X_i}(R^{-1}\tilde{\varepsilon})\|_{L_{\tau}^{N-2}H^{2\alpha}}
\notag \\
& \lesssim
\sum_{i=1}^{3} \| \lambda^{-2}\| \kappa(\frac{R}{\lambda})  \|_{\infty} \|\varphi(R/(2\lambda))X_i \partial_{X_i}(R^{-1}\tilde{\varepsilon})\|_{H^{2\alpha}} \|_{L_{\tau}^{N-2}}
\notag \\
& \quad +
\sum_{i=1}^{3} \| \lambda^{-2}\| (1-\Delta)^{\alpha} \kappa(\frac{R}{\lambda})  \|_{\infty} \| \varphi(R/(2\lambda))X_i\partial_{X_i}(R^{-1}\tilde{\varepsilon})\|_{L^{2}} \|_{L_{\tau}^{N-2}}.
\label{tw-pcs}
\end{align}
By the Hausdorf inequality,
\begin{align}
\| (1-\Delta)^{\alpha} \kappa(\frac{R}{\lambda})  \|_{\infty}
& =
\|  (1-\Delta)^{\alpha} \kappa(\frac{R}{\lambda})  \|_{\infty}
\notag \\
& \lesssim
 \|(1+|\zeta|^{2\alpha}) \lambda^3 \hat{\kappa}(\lambda \zeta) \|_{1}
\notag \\
& \lesssim
(1+\lambda^{-2\alpha}).
\label{hsd-bnd}
\end{align}
For the other terms, we use the support of $\varphi$:
\begin{align}
\|\varphi(R/(2\lambda))X_i \partial_{X_i}(R^{-1}\tilde{\varepsilon})\|_{H^{2\alpha}}
& =\|(1-\Delta)^{\alpha}\varphi(R/(2\lambda))X_i \partial_{X_i}(R^{-1}\tilde{\varepsilon})\|_{L^2}
\notag \\
& =\|[ (1-\Delta)^{\alpha},\varphi(R/(2\lambda))X_i ] \partial_{X_i}(R^{-1}\tilde{\varepsilon})\|_{L^2}
\notag \\
& \quad +
\|\varphi(R/(2\lambda))X_i (1-\Delta)^{\alpha} \partial_{X_i}(R^{-1}\tilde{\varepsilon})\|_{L^2}
\notag \\
& \lesssim \| \partial_{X_i}(R^{-1}\tilde{\varepsilon})\|_{L^2} 
+
\lambda\| (1-\Delta)^{\alpha} \partial_{X_i}(R^{-1}\tilde{\varepsilon})\|_{L^2}
\notag \\
& \lesssim \| R^{-1}\tilde{\varepsilon}\|_{H^{1+2\alpha}}
+
\lambda\| R^{-1}\tilde{\varepsilon}\|_{H^{1+2\alpha}}.
\label{cmm-wgh-trm}
\end{align}
With \eqref{tw-pcs}, \eqref{hsd-bnd}, and \eqref{cmm-wgh-trm}, and writing $t$ in terms of $\tau$
we find that
\begin{align}
\|\lambda^{-3}R\kappa(\frac{R}{\lambda})\lambda\partial_{R}(R^{-1}\tilde{\varepsilon})\|_{L_{\tau}^{N-2}H^{2\alpha}}
& \lesssim
\| \tau^{-1-1/\nu}\| R^{-1}\tilde{\varepsilon}\|_{H^{1+2\alpha}} \|_{L_{\tau}^{N-2}}.
\notag
\end{align}
The result follows since $\nu \in (0,1]$. 
\end{proof}

\begin{remark}
\label{lcl-glb}
 For the renormalization step, all of the computations were local. The perturbation step is a global computation, because it involves the nonlocal operator $\mathcal{F}$. We must take steps to reconcile these two complementary tools in the proof of Theorem \ref{main} below.
\end{remark}
We can now use Theorem \ref{renormalization}
followed by
Lemma \ref{crv-lps} to produce the required solution to \eqref{slv-m}. The steps are nearly identical to the conclusion section of \cite{Krg-Sch-Ttr}, with some small modifications due to the curved background.

\begin{proof}[Proof of Theorem \ref{main}]
From Theorem \ref{renormalization} we obtain an approximate solution $u_{2k+1}$ and an error term $e_{2k+1}$, but only within the light cone. As per Remark \ref{lcl-glb}, we must extend $u_{2k+1}$, and $e_{2k+1}$ before solving \eqref{prt-qtn}. We extend them to compactly supported functions $\tilde{u}_{2k+1}$, $\tilde{e}_{2k+1}$ of the same regularity in $r \leq 2t$.
 It follows from lemmas \ref{k-bnds}, \ref{nnl-lps}, \ref{crv-lps}, \ref{smt-ffc}, taking $N$ large enough, then $k$ large enough,
 and the contraction mapping theorem that there exists a solution $\varepsilon$ to \eqref{prt-qtn}. The function $u(t,r) = \tilde{u}_{2k+1}(t,r)+\varepsilon(t,r)$ is a solution to \eqref{slv-m-frs} within the light cone. Since
\begin{equation}
\lim_{t \rightarrow 0}\int_{K_t}(|\nabla (u - u_0)| + |u-u_0|^6) \mtr \,dr  =0,
\notag
\end{equation}
and
\begin{equation}
\lim_{t \rightarrow 0}\int_{K_t^c}(|\nabla u | + |u|^6) \mtr \,dr  =0,
\notag
\end{equation}
it is possible to choose a $t_0$ so that the set with $r \leq 3t_0$ is contained in $U^{\prime}$, and
\begin{equation}
\int_{r > t_0}(|\nabla (\tilde{u}_{2k+1}(r,t_0)+\varphi(r)\varepsilon(r,t_0)) | + |\tilde{u}_{2k+1}(r,t_0)+\varphi(r)\varepsilon(r,t_0)|^6) \mtr \,dr \lesssim \delta.
\notag
\end{equation}
We now solve the initial value problem for \eqref{slv-m} with initial data $(\tilde{u}_{2k+1}(r,t_0)+\varphi(r)\varepsilon(r,t_0), \partial_t\tilde{u}_{2k+1}(r,t_0)+\varphi(r)\partial_t\varepsilon(r,t_0))$.
 Call this solution $u$. By the transport properties of the wave equation, $u$ is equal to $u_{2k+1}+\varepsilon$ within the light  cone, so it blows up at the origin. Because of global well posedness for small energies, showing the energy outside the lightcone remains small will imply the solution must blow up only at the origin. Observe that
\begin{equation}
E(u) \approx E(W) + O(\delta).
\notag
\end{equation}
Therefore the energy outside the lightcone obeys
\begin{equation}
\int_{K_t^C} ( \frac{1}{2}|\partial_t u|^2 + \frac{1}{2}|\partial_r u|^2 - \frac{1}{6}|u|^6 ) \mtr \,dr \lesssim \delta.
\notag
\end{equation}
Since $\textrm{supp } {u} \subset U^{\prime}$, we also have the Sobolev estimate
\begin{equation}
\int_{r \ge t} |u|^6 \mtr \,dr \lesssim \left ( \int_{r \ge t} |\partial_r u|^2 \mtr \,dr  \right )^3.
\notag
\end{equation}
Therefore either
\begin{equation}
\int_{K_t^C} (|\partial_t u|^2 + |\partial_r u|^2 ) \mtr \,dr  \lesssim \delta,
\notag
\end{equation}
or
\begin{equation}
\int_{K_t^C} (|\partial_t u|^2 + |\partial_r u|^2  )\mtr \,dr  \gtrsim 1.
\notag
\end{equation}
The first alternative holds for the initial data. By continuity, it holds for all times.

\end{proof}
\noindent
\textbf{Acknowledgments:} We would like to thank Joachim Krieger for suggesting the problem and Roland Donninger for helpful discussions.

\end{document}